\documentclass[12pt]{amsart}
\usepackage{etex}
\usepackage{amsfonts, amssymb, latexsym}

\usepackage{amscd,amssymb}
\usepackage{verbatim}
\usepackage{pstricks}
\usepackage{pst-node}
\usepackage[mathscr]{euscript}
 \usepackage{tikz}
\usetikzlibrary{matrix,arrows}
\usepackage[normalem]{ulem}

\newsymbol\pp 1275
\usepackage{tikz}
\usetikzlibrary{matrix,arrows,backgrounds,shapes.misc,shapes.geometric,patterns,calc,positioning}

\usepackage[colorlinks=true, pdfstartview=FitV, linkcolor=blue, citecolor=blue, urlcolor=blue]{hyperref}

\usepackage{fullpage}
\usepackage{graphicx}
\usepackage{enumerate}

\def\VR{\kern-\arraycolsep\strut\vrule &\kern-\arraycolsep}
\def\vr{\kern-\arraycolsep & \kern-\arraycolsep}

\newtheorem{theorem}{Theorem}

\newtheorem{lemma}[theorem]{Lemma}
\newtheorem{prop}[theorem]{Proposition}
\newtheorem{corollary}[theorem]{Corollary}
\theoremstyle{definition}
\newtheorem{definition}[theorem]{Definition}

\newtheorem{rmk}[theorem]{Remark}
\newenvironment{remark}[1][]{\begin{rmk}[#1]\pushQED{\qed}}{\popQED \end{rmk}}

\newtheorem*{rmknonum}{Remark}

\newtheorem{obs}[theorem]{Observation}

\newtheorem{ex}[theorem]{Example}
\newenvironment{example}[1][]{\begin{ex}[#1]\pushQED{\qed}}{\popQED \end{ex}}
\newcommand{\pctext}[2]{\text{\parbox{#1}{\centering #2}}}

\newcommand{\Hom}{\operatorname{Hom}}

\newcommand{\ext}{\operatorname{ext}}
\newcommand{\rep}{\operatorname{rep}}
\newcommand{\Proj}{\operatorname{Proj}}

\newcommand{\SI}{\operatorname{SI}}

\newcommand{\GL}{\operatorname{GL}}
\newcommand{\PGL}{\operatorname{PGL}}

\newcommand{\ZZ}{\mathbb Z}

\newcommand{\PP}{\mathbb P}

\newcommand{\Ker}{\operatorname{Ker}}

\newcommand{\res}{\operatorname{res}}
\newcommand{\Id}{\operatorname{Id}}

\newcommand{\supp}{\operatorname{supp}}
\newcommand{\Mat}{\operatorname{Mat}}

\newcommand{\Spec}{\operatorname{Spec}}

\newcommand{\filt}{\mathcal Filt}

\newcommand{\ddim}{\operatorname{\mathbf{dim}}}
\newcommand{\cc}{\operatorname{\mathbf{c}}}
\newcommand{\dd}{\operatorname{\mathbf{d}}}
\newcommand{\ee}{\operatorname{\mathbf{e}}}

\newcommand{\ff}{\operatorname{\mathbf{f}}}

\newcommand{\m}{\mathfrak{m}}

\newcommand{\R}{\operatorname{\mathcal{R}}}
\newcommand{\M}{\operatorname{\mathcal{M}}}

\newcommand\restr[2]{{% we make the whole thing an ordinary symbol
  \left.\kern-\nulldelimiterspace % automatically resize the bar with \right
  #1 % the function
  \vphantom{\big|} % pretend it's a little taller at normal size
  \right|_{#2} % this is the delimiter
  }}

\newcommand{\setst}[2]{\{#1\ \mid \ #2 \}}

\newcommand{\key}[1]{\emph{#1}}

\begin{document}
\title{Decomposing moduli of representations of finite-dimensional algebras}
\author{Calin Chindris}
\address{University of Missouri-Columbia, Mathematics Department, Columbia, MO, USA}
\email[Calin Chindris]{chindrisc@missouri.edu}

\author{Ryan Kinser}
\address{University of Iowa, Department of Mathematics, Iowa City, IA, USA}
\email[Ryan Kinser]{ryan-kinser@uiowa.edu}

\date{\today}
\bibliographystyle{amsalpha}
\subjclass[2010]{16G20, 14L24, 14D20}
\keywords{moduli spaces, representations of algebras, quivers}

\begin{abstract}
Consider a finite-dimensional algebra $A$ and any of its moduli spaces $\mathcal{M}(A,\mathbf{d})^{ss}_{\theta}$ of representations. We prove a decomposition theorem which relates any irreducible component of $\mathcal{M}(A,\mathbf{d})^{ss}_{\theta}$ to a product of simpler moduli spaces via a finite and birational map. Furthermore, this morphism is an isomorphism when the irreducible component is normal. As an application, we show that the irreducible components of all moduli spaces associated to tame (or even Schur-tame) algebras are rational varieties.
\end{abstract}

\maketitle
\setcounter{tocdepth}{1}
\tableofcontents

\section{Introduction}
\subsection{Context and motivation}
Throughout the paper, $K$ denotes an algebraically closed field of characteristic zero, and all algebras $A$ are assumed to be associative and finite-dimensional over $K$.  Since we are interested in collections of representations or modules over such an algebra, there is no loss of generality in assuming $A$ is basic and taking $A= KQ/I$ for some quiver $Q$ and admissible ideal $I$. By a slight abuse of terminology, we say ``representations of $A$'' to mean ``representations of $Q$ satisfying a set of admissible relations generating $I$''.

In this paper, we study representations of algebras within the general framework of Geometric Invariant Theory (GIT).  
This is motivated by the fact that, for many algebras, their representations cannot be classified in any convenient list form.  More precisely, the theory of $A$-modules is known to be undecidable for many algebras (it is conjectured to be undecidable for all wild algebras \cite[Ch.~17]{prest}).
In the geometric approach, algebraic varieties known as \emph{moduli spaces} are constructed to parametrize families of representations, and we study the structure of these varieties.  Besides their role in the representation theory of algebras, these moduli spaces and their framed versions also naturally arise in other mathematical contexts, such as moduli of sheaves and vector bundles \cite{Bondal1,Bondal2,ACK1}, toric varieties \cite{AltHil,Hil2,CraSmi}, Marsden-Weinstein (or symplectic) quotients \cite{CB-MW,CB-MW2,K}, and quantum groups \cite{Nakajima96,Reineke03,Reineke08}.  They also have connections with Donaldson-Thomas invariants, cluster varieties, and mathematical physics  \cite{Reineke11,KonSoi11,JoySon12,Mozgovoy1,ABCH,ACCEetal,cordova1,Bridgeland16}. 
%\cite{BMW,Stupariu}

We briefly recall just enough of the main ideas here for motivation, with more detailed background found in Section \ref{sec:background}.
A choice of weight $\theta$, which is nothing more than an assignment of an integer to each vertex of $Q$, determines subcategories $\rep(A)^s_\theta \subset \rep(A)^{ss}_\theta \subset \rep(A)$ of \emph{$\theta$-stable} and \emph{$\theta$-semistable} representations of $A$, respectively.  The category $\rep(A)^{ss}_\theta$ is abelian with $\rep(A)^s_\theta$ being precisely the simple objects, so that every $\theta$-semistable representation has a well-defined collection of $\theta$-stable composition factors, by the Jordan-H\"older theorem.

For a fixed dimension vector $\dd$, the sets of $\theta$-stable and $\theta$-semistable representations form open (possibly empty) subvarieties  $\rep(A,\dd)^s_\theta \subset \rep(A,\dd)^{ss}_\theta \subset \rep(A,\dd)$. GIT gives a procedure for taking a ``quotient'' of $\rep(A,\dd)^{ss}_{\theta}$ by the base change group $\PGL(\dd)$.  The GIT quotient $\M(A,\dd)^{ss}_\theta$ parametrizes its closed orbits, which are in bijection with semi-simple objects in $\rep(A)^{ss}_\theta$ of dimension vector $\dd$.
This procedure works equally well for any $\GL(\dd)$-invariant, closed subvariety $C \subseteq \rep(A,\dd)$, giving us a geometric method for studying ``families of representations'' of $A$.  Our main result (Theorem \ref{main-thm}) makes precise the sense in which families of $\theta$-semistable representations are ``controlled'' by the families of their $\theta$-stable composition factors.

Additional motivation for this paper is a program aimed at finding geometric characterizations of the representation type of algebras.  Although arbitrary projective varieties can arise as moduli spaces of representations of algebras \cite{Hil,HZ2},  representation theoretic properties of a given algebra can impose constraints on the moduli spaces.  For example, we show in Corollary \ref{cor:tamerational} that all moduli spaces associated to Schur-tame algebras are rational varieties. This line of research has attracted a lot of attention, see for example
\cite{BleChiHui-Zim15,Bob1,Bob5, Bob4,BS1,CC15,CC6,CC9,ChiKli16,CC12, Carroll1, Domo2, GeiSch, Rie, Rie-Zwa-1, Rie-Zwa-2, SW1}.
% \cite{Weist1, Weist2}.
 
\begin{rmknonum}
In the last section, one application and some examples of our main result are illustrated for \emph{tame algebras}. Such algebras, roughly speaking, are less complex than arbitrary algebras in that almost all of their indecomposables of a given dimension can be parametrized by finitely many families of one $K$-parameter.
We offer a few observations here on moduli of representations of tame algebras.
\begin{enumerate}[(a)]%{\setlength\itemindent{0ex}}
\item There are only very few tame algebras whose representations are actually classified, for example hereditary or special biserial algebras (or algebras close to these).
This is because there are many techniques for showing that an algebra is tame without actually classifying \emph{all} of its indecomposable representations, for example degeneration of algebra structure \cite{CB95, Geiss95}, tilting techniques \cite{R3}, or vector space category methods \cite{Ringel80}.
\item Even when the indecomposable representations of a tame algebra are classified by methods of algebra and combinatorics, this does \emph{not} give a \emph{geometric} description of the associated moduli spaces of representations or their irreducible components.  While it is straightforward to see that moduli spaces of stable representations are rational projective curves, it is not known whether they are always smooth, so it is unknown whether such a moduli space is isomorphic to $\PP^1$ in general.
\item Even when a geometric description of moduli spaces of stable representations is known for an algebra, this does not entail a description of general moduli spaces of semistable representations.  
A case in point can be found in \cite{CarChiKinWey2017}, where irreducible components of moduli of representations of special biserial algebras are studied.  Indecomposable representations of special biserial algebras are some of the most thoroughly studied among tame algebras, admitting a very explicit combinatorial description.  Nevertheless, to determine the isomorphism class of arbitrary moduli of semistable representations, we still need to apply the main result of the present work in conjunction with connections we establish between representation varieties and affine Schubert varieties (via work of Lusztig \cite{Lusztig}).
\end{enumerate}
Finally, we stress that while tame algebras are addressed in one example and one corollary, our primary interest and main result of this work is in the generality of arbitrary finite-dimensional associative algebras.
\end{rmknonum}

\subsection{Statement of main result}
Let $A$ be an arbitrary finite-dimensional algebra and $C \subseteq \rep(A, \dd)$ a $\GL(\dd)$-invariant, irreducible, closed subvariety.  A $\theta$-stable decomposition $C = m_1 C_1 \pp \cdots \pp m_r C_r$ records a collection of $\theta$-stable irreducible components $C_i \subseteq \rep(A, \dd_i)$ which parametrize the $\theta$-stable composition factors appearing in general $\theta$-semistable representations in $C$ along with multiplicities $m_i$ (see Definition \ref{def:thetastable}).  

Our main result is the following decomposition theorem. It describes each irreducible component of a moduli of representations in terms of the moduli spaces of the components of a $\theta$-stable decomposition and their multiplicities.

\begin{theorem} \label{main-thm} Let $A$ be a finite-dimensional algebra and let $C \subseteq \rep(A,\dd)^{ss}_{\theta}$ be a $\GL(\dd)$-invariant, irreducible, closed subvariety.  Let $C=m_1C_1\pp \ldots \pp m_r C_r$ be a $\theta$-stable decomposition of $C$ where $C_i \subseteq \rep(A,\dd_i)$, $1 \leq i \leq r$, are pairwise distinct $\theta$-stable irreducible components, and define $\widetilde{C} = \overline{C_1^{\oplus m_1} \oplus \cdots \oplus C_r^{\oplus m_r}}$.
\begin{enumerate}[(a)]
\item 
If $\M(C)^{ss}_{\theta}$ is an irreducible component of $\M(A,\dd)^{ss}_{\theta}$, then
$$\M(\widetilde{C})^{ss}_{\theta}=\M(C)^{ss}_{\theta}.$$
\item If $C_1$ is an orbit closure, then $$\M(\overline{C_1^{\oplus m_1} \oplus \cdots \oplus C_r^{\oplus m_r}})^{ss}_{\theta} \simeq \M(\overline{C_2^{\oplus m_2} \oplus \cdots \oplus C_r^{\oplus m_r}})^{ss}_{\theta}.$$ 
\item Assume now that none of the $C_i$ are orbit closures.
Then there is a natural morphism
\[
\Psi\colon  S^{m_1}(\mathcal{M}(C_1)^{ss}_{\theta}) \times \ldots \times S^{m_r}(\mathcal{M}(C_r)^{ss}_{\theta})  \to \M(\widetilde{C})^{ss}_{\theta}
\]
which is finite, and birational. In particular, if $\M(\widetilde{C})^{ss}_{\theta}$ is normal then $\Psi$ is an isomorphism. 
\end{enumerate}
\end{theorem}

To analyze moduli spaces of a given algebra, we typically proceed as follows: by (a), we may assume that a general point of $C$ is simply the direct sum of its $\theta$-stable composition factors.  Then repeatedly applying (b) allows us to get rid of all the orbit closures that occur in a $\theta$-stable decomposition. These are very useful reductions since it can be easier to check the normality condition in (c) under these much more restrictive conditions on $C$. In (c) we see that multiplicities in the $\theta$-stable decomposition simply contribute symmetric powers to the overall moduli space, at least up to birational equivalence; here, recall that the \key{$m^{th}$ symmetric power} $S^m(X)$ of a variety $X$ is the quotient of $\prod_{i=1}^m X$ by the action of the symmetric group on $m$ elements which permutes the coordinates.   We also show in Example \ref{ex:notiso} that $\Psi$ is not an isomorphism in general.

\subsection{Relation to existing literature}
Here we briefly survey the relation between our results and existing literature.
The notion of $\theta$-stable decomposition was introduced by Derksen and Weyman \cite{DW2} for the case that $A=KQ$ where $Q$ is acyclic (so that all $\rep(A, \dd)$ are just vector spaces).  An extension to $\GL(\dd)$-invariant irreducible subvarieties $C \subseteq \rep(A,\dd)$ when $A$ is an arbitrary algebra was given in \cite{CC9, CC10, CC15}.

Theorem \ref{main-thm}(b) and the ingredients going into it are inspired by Bobi{\'n}ski's work \cite{Bob5} which assumes that $A=KQ/I$ where $Q$ is acyclic. To deal with the general case, we bring into the picture Schofield's double quiver $\widehat{Q}$ of $Q$, which is acyclic by construction. In Proposition \ref{FFT-thm}, we prove an extension of the First Fundamental Theorem for semi-invariants of quivers, explaining how Schofield determinantal semi-invariants of $\widehat{Q}$ relate to those of $A$. This plays a key role in the proof of Theorem \ref{main-thm}(b). The overall proof of this part of our main result builds on the work of  Bobi{\'n}ski  in \textit{ibid.},  Derksen-Weyman \cite{DW1}, Igusa-Orr-Todorov-Weyman \cite{IOTW}, Domokos \cite{Domo2}, and Schofield-van den Bergh \cite{SVB}. 

Theorem \ref{main-thm}(c) is a generalization to arbitrary algebras of a theorem of Derksen and Weyman, who proved it in the case of $A=KQ$ with $Q$ acyclic  \cite[Theorem 3.16]{DW2}. Intermediate improvements of the Derksen-Weyman result can be found in \cite[Theorem 1.4]{CC10} and \cite[Proposition 7]{CC15}. All these earlier generalizations assume more restrictive normality conditions, and exclude the possibility of more than one $\theta$-stable irreducible component of the same dimension vector (i.e. assume $\dd_i \neq \dd_j$ for $i\neq j$ in the statement of the theorem). Being able to get rid of this ``separation'' condition is especially important for applications. A first example of this can be found in Example \ref{ex:families}. Generalizing this in \cite{CarChiKinWey2017}, we use Theorem \ref{main-thm} in an essential way to show that the irreducible components of any moduli space associated to arbitrary special biserial algebras are products of projective spaces. 

In Corollary \ref{cor:tamerational},  we show that moduli spaces of ``Schur-tame'' or ``brick-tame'' algebras (a class which includes all tame algebras but also many wild algebras) are always rational varieties. This continues a long line of results on birational classification of moduli spaces of representations; see for example work of Ringel \cite{R4} or Schofield \cite{Schofield01} on the case of $A=KQ$.  

We also note that a decomposition theorem due to Crawley-Boevey \cite{CB-MW}  for symplectic reductions (or Marsden-Weinstein reductions), in the setting where $A=KQ$, is similar in form to our main result specialized to that case. He also showed that these varieties are always normal \cite{CB-MW2}. A connection between symplectic reductions and moduli spaces is discussed in \cite[\S6]{K}.

\subsection*{Acknowledgements}  
We wish to thank Grzegorz Bobi\'nski and Alastair King for discussions that led to improvements of our paper.  We are especially thankful to Harm Derksen for clarifying discussions on some of the results in \cite{DW2}. The first author (C.C.) was supported by the NSA under grant H98230-15-1-0022. 

\section{Background}\label{sec:background}
\subsection{Representation varieties} Due to a fundamental observation of Gabriel, the category of modules over any finite-dimensional unital, associative $K$-algebra $A$ is equivalent to the category of modules over a quotient of the path algebra of a finite quiver.  More precsiely, there exist a quiver $Q$ (uniquely determined by $A$) and an ideal $I$ of $KQ$ generated by a collection $\mathcal{R}$ of linear combinations of paths of length at least 2, such that $A$ is Morita equivalent to $KQ/I$.  Therefore, we always implicitly identify algebras $A$ with quotients of path algebras throughout, and representations of $A$ with representations of the corresponding quiver which satisfy the relations in $\mathcal{R}$.  More background on representations of algebras and quivers can be found in \cite{AS-SI-SK,Schiffler:2014aa}.

To fix notation, we write $Q_0$ for the set of vertices of a quiver $Q$, and $Q_1$ for its set of arrows, while $ta$ and $ha$ denote the tail and head of an arrow $ta \xrightarrow{a} ha$. A \key{representation} $M$ of $Q$ of dimension vector $\dd \in \ZZ^{Q_0}_{\geq 0}$ assigns a $\dd(x)$-dimensional vector space $M(x)$ to each $x\in Q_0$, and to each $a \in Q_1$ a choice of linear map $M(a) \colon M(ta) \to M(ha)$.  The \key{Euler-Ringel bilinear form} of $Q$ on the space $\ZZ^{Q_0}\times \ZZ^{Q_0}$ is denoted by 
\[
\langle\dd, \dd'\rangle_Q=\sum_{x \in Q_0} \dd(x)\dd'(x) - \sum_{a \in Q_1} \dd(ta)\dd'(ha).
\]

For a dimension vector $\dd$, the affine \emph{representation variety} $\rep(A,\dd)$ parametrizes the $\dd$-dimensional representations of $(Q,\R)$ along with a fixed basis. So we have:
$$
\rep(A,\dd):=\{M \in \prod_{a \in Q_1} \Mat_{\dd(ha)\times \dd(ta)}(K) \mid M(r)=0, \text{~for all~} r \in
\R \}.
$$
Under the action of the change of base group $\GL(\dd):=\prod_{x\in Q_0}\GL(\dd(x),K)$, the orbits in $\rep(A,\dd)$ are in one-to-one correspondence with the isomorphism classes of $\dd$-dimensional representations of $(Q, \R)$. For more background on module and representation varieties, see surveys such as \cite{Bon5,Zwarasurvey,HZsurvey}.
We remark once and for all that we only work at the level of varieties in this paper, ignoring reducedness and other scheme-theoretic concerns throughout.

In general, $\rep(A, \dd)$ does not have to be irreducible. Let $C$ be an irreducible component of $\rep(A, \dd)$. We say that $C$ is \emph{indecomposable} if $C$ has a nonempty open subset of indecomposable representations. We say that $C$ is a \emph{Schur component} if $C$ contains a Schur representation, in which case $C$ has a nonempty open subset of Schur representations; in particular, any Schur component is indecomposable.

Given a collection of subvarieties $\{C_i \subseteq \rep(A, \dd_i)\}_{i=1}^r$, let $\dd = \sum_i \dd_i$, so we have the subvariety $\prod_i C_i \subseteq \rep(A, \dd)$.  We define their direct sum $\overline{C_1 \oplus \ldots \oplus C_r}$ to be the closure of $\GL(\dd) \cdot \prod_i C_i$.
It was shown by de la Pena in \cite{delaP} and Crawley-Boevey and Schr{\"o}er in \cite[Theorem 1.1]{C-BS} that any irreducible component $C \subseteq \rep(A,\dd)$ satisfies a Krull-Schmidt type decomposition
$$
C=\overline{C_1 \oplus \ldots \oplus C_r}
$$
for some indecomposable irreducible components $C_i \subseteq \rep(A,\dd_i)$ with $\sum \dd_i=\dd$. Moreover, $C_1, \ldots, C_r$ are uniquely determined by this property. 

\subsection{Semi-invariants}
The first ingredient to constructing moduli spaces of quiver representations are spaces of semi-invariants, which we review here. For each rational character $\chi\colon \GL(\dd) \to K^*,$ the vector space
$$
\SI(A,\dd)_{\chi}=\lbrace f \in K[\rep(A,\dd)] \mid g \cdot f= \chi(g)f \text{~for all~}g \in \GL(\dd)\rbrace
$$
is called the \key{space of semi-invariants} on $\rep(A,\dd)$ of \key{weight} $\chi$. 
For a $\GL(\dd)$-invariant closed subvariety $C \subseteq \rep(A,\dd)$, we similarly define the space $\SI(C)_{\chi}$ of semi-invariants of weight $\chi \in X^{\star}(\GL(\dd))$. 

Note that any $\theta \in \ZZ^{Q_0}$ defines a rational character $\chi_{\theta}\colon\GL(\dd) \to K^*$ by 
\begin{equation}\label{eq:chitheta}
\chi_{\theta}((g(x))_{x \in Q_0})=\prod_{x \in Q_0}\det g(x)^{\theta(x)}.
\end{equation}
In this way, we get a natural epimorphism $\ZZ^{Q_0} \to X^{\star}(\GL(\dd))$, and we refer to either $\theta$ or $\chi_{\theta}$ as an \emph{integral weight of $Q$ (or $A$)}. In case $\dd$ is a \emph{sincere} dimension vector (i.e., $\dd(x) \neq 0$ for all $x \in Q_0$), this epimorphism is an isomorphism which allows us to identify $\ZZ^{Q_0}$ with $X^{\star}(\GL(\dd))$. 
From now on, we us assume that all of our integral weights $\theta$ are so that $\chi_{\theta}$ is a non-trivial character of $\GL(\dd)$, i.e. the restriction of $\theta$ to the support of $\dd$ is not zero, and denote by $G_{\theta} \trianglelefteq \GL(\dd)$ the kernel of $\chi_{\theta}$.  Let $C$ be a $\theta$-semistable $\GL(\dd)$-invariant, irreducible, closed subvariety of $\rep(A,\dd)$. Then we have a decomposition of the invariant ring
\begin{equation}\label{eq:KCGtheta}
K[C]^{G_{\theta}}=\bigoplus_{m \geq 0} \SI(C)_{m \theta}.
\end{equation}

\subsection{Moduli spaces of representations}\label{moduli-sec}
Fix an integral weight $\theta \in \ZZ^{Q_0}$ of $A$; we define its evaluation on dimension vectors $\dd$ by $\theta(\dd) = \sum_{x \in Q_0} \theta(x)\dd(x)$.  

Following King \cite{K}, a representation $M$ of $A$ is said to be \emph{$\theta$-semistable} if $\theta(\ddim M)=0$ and $\theta(\ddim M')\leq 0$ for all subrepresentations $M' \leq M$. We say that $M$ is \emph{$\theta$-stable} if $M$ is nonzero, $\theta(\ddim M)=0$, and $\theta(\ddim M')<0$ for all subrepresentations $0 \neq M' < M$. Finally, we call $M$ a \key{$\theta$-polystable} representation if $M$ is a direct sum of $\theta$-stable representations.  
It was noted by King that the collection of $\theta$-semistable representations of $A$ forms a full abelian subcategory of $A$ in which the $\theta$-stable representations are precisely the simple objects; in particular, Hom spaces between $\theta$-stable representations have dimension one or zero.  Two $\theta$-semistable representations are said to be \key{$S$-equivalent} if they have the same collection of $\theta$-stable composition factors (counted with multiplicity).

Now, let $\dd$ be a dimension vector of $A$ and consider the (possibly empty) open subsets
$$\rep(A,\dd)^{ss}_{\theta}=\{M \in \rep(A,\dd)\mid M \text{~is~}
\text{$\theta$-semistable}\}$$
$$\rep(A,\dd)^s_{\theta}=\{M \in \rep(A,\dd)\mid M \text{~is~}
\text{$\theta$-stable}\}.$$
We say that $\dd$ is a \emph{$\theta$-(semi-)stable dimension vector} of $A$ if $\rep(A,\dd)^{(s)s}_{\theta} \neq \emptyset$. 
A GIT quotient of $\rep(A,\dd)^{ss}_{\theta}$ by the action of $\PGL(\dd)$ is constructed by King in \cite{K}, where $\PGL(\dd)=\GL(\dd)/K^*$ with $K^*$ identified with scalar multiples of the identity of the group $\GL(\dd)$. 
This quotient is defined as
$$
\M(A,\dd)^{ss}_{\theta}:=\Proj\left(\bigoplus_{m \geq 0}\SI(A,\dd)_{m\theta}\right);
$$
it is a projective variety which is a coarse moduli space for $\theta$-semistable representations of dimension vector $\dd$, up to $S$-equivalence. Equivalently, it parametrizes $\theta$-polystable representations of dimension vector $\dd$, or more geometrically,  the $\GL(\dd)$-orbits in $\rep(A,\dd)$ which are closed in $\rep(A,\dd)^{ss}_{\theta}$. Moreover, there is a (possibly empty) open subset $\M(A,\dd)^s_{\theta}$ of $\M(A,\dd)^{ss}_{\theta}$ which is a geometric quotient of $\rep(A,\dd)^s_{\theta}$ by $\PGL(\dd)$.

For a given $\GL(\dd)$-invariant closed subvariety $C$ of $\rep(A,\dd)$, we similarly define $C^{ss}_{\theta}, C^s_{\theta}$, $\M(C)^{ss}_{\theta}$, and $\M(C)^s_{\theta}$. We say that $C$ is a \emph{$\theta$-(semi)stable subvariety} if $C^{(s)s} \neq \emptyset$. 
The invariant ring $K[C]^{G_\theta}$ is by definition the homogeneous coordinate ring of $\M(C)^{ss}_\theta$.
The following two commutative diagrams summarize the relation between the various rings and spaces we consider, with justifications given below.
\begin{equation}\label{eq:invariant diagrams}
\vcenter{\hbox{\begin{tikzpicture}
\node (1) at (0,2) {$K[\rep(A,\dd)]^{G_{\theta}}$};
\node (2) at (0,0) {$K[C]^{G_{\theta}}$};
\node (3) at (4,2) {$K[\rep(A,\dd)]$};
\node (4) at (4,0) {$K[C]$};
\path[->>]
(1) edge node[left] {$\res_C^{G_{\theta}}$} (2)
(3) edge node[left] {$\res_C$} (4);
\path[right hook-latex]
(1) edge (3)
(2) edge (4);
\end{tikzpicture}}}
\qquad
\vcenter{\hbox{\begin{tikzpicture}
\node (1) at (0,2) {$\M(A,\dd)^{ss}_{\theta}$};
\node (2) at (0,0) {$\M(C)^{ss}_{\theta}$};
\node (3) at (4,2) {$\rep(A,\dd)^{ss}_{\theta}$};
\node (4) at (4,0) {$C^{ss}_{\theta}$};
\path[right hook-latex]
(2) edge node[left] {$$} (1)
(4) edge node[left] {$$} (3);
\path[->>]
(3) edge node[above] {$\pi$} (1)
(4) edge node[above] {$\restr{\pi}{C^{ss}_{\theta}}$} (2);
\end{tikzpicture}}}
\end{equation}

The map $\res_C^{G_{\theta}}$ inherits surjectivity from $\res_C$ since $G_{\theta}$ is linearly reductive in characteristic zero \cite[Cor. 2.2.9]{Harm-book-02}. Appying $\Proj$ to this surjective homomorphism of graded algebras gives rise to the closed embedding of moduli spaces in the diagram at right.

The points of $\M(C)^{ss}_{\theta}$ correspond bijectively to the (isomorphism classes of) $\theta$-polystable representations in $C$. Indeed, each fiber of $\pi\colon C^{ss}_{\theta} \to \M(C)^{ss}_{\theta}$ contains a unique $\GL(\dd)$-orbit which is closed in $C^{ss}_{\theta}$. On the other hand, as proved by King in \cite[Proposition 3.2(i)]{K}, these orbits are precisely the isomorphism classes of $\theta$-polystable representation in $C$. In fact, for any $M \in C^{ss}_{\theta}$, there exists a $1$-parameter subgroup $\lambda \colon K^* \to G_\theta$ such that $\widetilde{M}:=\lim_{t \to 0} \lambda(t)M$ exists and is the unique, up to isomorphism, polystable representation in $\overline{\GL(\dd)M} \cap C^{ss}_{\theta}$.

\subsection{$\theta$-stable decompositions}
We now introduce the notion of a $\theta$-stable decomposition, which is a slight generalization of the definition in \cite[Section 3C]{CC10}.

\begin{definition}\label{def:thetastable}
Let $C$ be a $\GL(\dd)$-invariant, irreducible, closed subvariety of $\rep(A,\dd)$, and assume $C$ is $\theta$-semistable. Consider a collection $(C_i \subseteq \rep(A,\dd_i))_i$ of $\theta$-stable irreducible components such that $C_i \neq C_j$ for $i \neq j$, along with a collection of multiplicities $(m_i \in \ZZ_{>0})_i$.
We say that $(C_i, m_i)_i$ is a \key{$\theta$-stable decomposition of $C$} if, for a general representation $M \in C^{ss}_{\theta}$, its corresponding $\theta$-polystable representation $\widetilde{M}$ is in $C_1^{\oplus m_1} \oplus \cdots \oplus C_l^{\oplus m_l} $, and write
\begin{equation}\label{eq:thetastabledef}
C=m_1C_1\pp \ldots \pp m_l C_l.
\end{equation}
\end{definition}
This follows the notation of \cite{DW2}, whose definition of $\theta$-stable decomposition agrees with ours in the case that $A=KQ$ for an acyclic quiver $Q$.

\begin{prop} Any $\GL(\dd)$-invariant, irreducible, closed subvariety $C$ of $\rep(A,\dd)$ with $C^{ss}_{\theta} \neq \emptyset$ admits a $\theta$-stable decomposition.
\end{prop}

\begin{proof} Given any collection of irreducible components $\mathbf{C} = (C_i \subseteq \rep(A, \dd'_i))_{i=1}^R$ such that $(C_i)^s_{\theta} \neq \emptyset$, consider the (possibly empty) locus 
\[
\mathcal{F}(\mathbf{C}):= \setst{M \in C^{ss}_{\theta}}{\text{there exist }M_i \in (C_i)^s_{\theta}\text{ with }\widetilde{M}\simeq \bigoplus_i M_i}.
\]
Since every $\theta$-semistable representation has a filtration with $\theta$-stable composition factors, the non-empty sets of the form $\mathcal{F}(\mathbf{C})$ provide a finite cover $C^{ss}_\theta$. Therefore at least one $\mathcal{F}(\mathbf{C})$ is dense in $C^{ss}_\theta$ since $C^{ss}_{\theta}$ is irreducible.

\smallskip
\noindent
\textbf{Claim:} For any collection $\mathbf{X}=(X_i \subseteq \rep(A, \ff_i))_{i=1}^R$ of $\GL(\ff_i)$-invariant constructible subsets of $\rep(A,\ff_i)$, the subset 
\begin{alignat*}{2} 
\filt(\mathbf X) &= \Biggl\{ M \in \rep(A,\ff) &&\;\Bigg|\; \pctext{3in}{$\exists$ a filtration of representations  $0=M_0<M_1<\ldots < M_R=M$  such that $M_i/M_{i-1}$ is isomorphic to a representation in $C_i$}\Biggr\} \\
\end{alignat*}
is a constructible subset of $\rep(A,\ff)$ where $\ff=\ff_1+\ldots+\ff_R$.
\smallskip

Given this claim, any of the $\mathcal{F}(\mathbf{C})$ above is constructible since it is a union over the symmetric group $S_R$ of constructible subsets of the form $\filt(\mathbf X_{\sigma})$ where $X_{\sigma}=((C_{\sigma(i)})^s_{\theta} \subseteq \rep(A, \dd'_{\sigma(i)}))_{i=1}^R$ and $\sigma \in S_R$. So, we get at least one $\mathcal{F}(\mathbf{C})$ which is both contructible and dense in $C^{ss}_{\theta}$. Therefore it must contain an open and dense subset of $C^{ss}_{\theta}$, proving the existence of a $\theta$-stable decomposition of $C$.

To prove the claim, we first note that for any $\mathbf X$, $\filt(\mathbf X)=\filt(X', X_R)$ where $X'=\filt((X_i)_{i=1}^{R-1})$. So, it comes down to proving the claim for $R=2$. This can be easily checked by considering the morphism of varieties
\begin{align*}
f:\GL(\ff)\times \rep(Q,\ff_1)\times \rep(Q,\ff_2)\times \prod_{a \in Q_1} \Mat_{\ff_1(ha)\times \ff_2(ta)}(K) &\to \rep(Q,\ff)\\
(g, M', M'', (X(a))_{a \in Q_1}) &\to g \cdot \left(\begin{matrix}
M'(a) & X(a)\\
0&M''(a)
\end{matrix}  \right)_{a \in Q_1}
\end{align*}
Then $\filt(\mathbf X)=f(\GL(\ff), X_1, X_2, \prod_{a \in Q_1} \Mat_{\ff_1(ha)\times \ff_2(ta)}(K))$ is constructible by Chevalley's theorem (see for example \cite{Hart77}).
\end{proof}

\begin{proof}[Proof of Theorem \ref{main-thm}{(a)}] 
We get $\M(C)^{ss}_\theta \subseteq \M(\widetilde{C})^{ss}_{\theta}$ since the $\theta$-stable composition factors of a general element of $C^{ss}_\theta$ lie in $\widetilde{C}^{ss}_{\theta}$. But $\M(C)^{ss}_\theta$ is assumed to be an irreducible component of $\M(A,\dd)^{ss}_{\theta}$, so it must be equal to the closed, irreducible subvariety $\M(\widetilde{C})^{ss}_{\theta}$.
\end{proof}

\section{Removing orbit closures}
In this section, we first make some technical advances necessary to work with algebras arising from quivers with oriented cycles.   We then use these to show that direct summands of $C$ which are orbit closures can be thrown out without changing the geometry of the moduli space.

\subsection{Schofield semi-invariants when $Q$ has oriented cycles}
Our goal in this section is to show that given a collection of (nonzero) spaces of semi-invariants $\SI(C_1)_{\theta}, \ldots, \SI(C_n)_{\theta}$ of common weight $\theta$, there is a common locus of representations whose associated Schofield semi-invariants span each $\SI(C_i)_{\theta}$, $1 \leq i \leq n$. This is proved by Bobi{\'n}ski in \cite[Lemma 4.6]{Bob5}, under the assumption that $Q$ is acyclic. In what follows, we explain how to adapt Bobi{\'n}ski's proof strategy to the general case. For arbitrary quivers $Q$, we need to work with the so-called Schofield's double quiver of $Q$, which is acyclic by construction. 

The \key{double quiver} of $Q$ is the bipartite quiver $\widehat{Q}$ defined as follows.  The set of vertices of $\widehat{Q}$ is $\widehat{Q}_0=Q_0\times \{0,1\}$.
For convenience, we denote the vertices in $\widehat{Q}$ corresponding to a vertex $v \in Q_0$ by $v_0$ and $v_1$. The set of arrows of $\widehat{Q}$ is 
$$\widehat{Q}_1=\{c_v:v_0 \to v_1\}_{v \in Q_0}\cup \{\hat{a}:(ta)_0 \to (ha)_1\}_{a \in Q_1}.$$ 
We have the natural embedding $\tau: \rep(Q) \to \rep(\widehat{Q})$, sending $V$ to $\widehat{V}$ defined as follows: $\widehat{V}(v_0)=\widehat{V}(v_1)=V(v)$ and $\widehat{V}(c_v)=\Id_{V(v)}$ for every $v \in Q_0$, and $\widehat{V}(\hat{a})=V(a)$ for every $a \in Q_1$.
For $\dd \in \ZZ^{Q_0}$, define $\widehat{\dd} \in \ZZ^{\widehat{Q}_0}$ by $\widehat{\dd}(v_0)=\widehat{\dd}(v_1)=\dd(v)$ for all $v \in Q_0$. If $\dd \in \ZZ^{Q_0}_{\geq 0}$ is a dimension vector of $Q$, denote by $\tau_{\dd}\colon\rep(Q,\dd) \to \rep(\widehat{Q},\widehat{\dd})$ the closed embedding induced by $\tau$.

Now, let $\widehat{\cc}$ and $\widehat{\dd}$ be two dimension vectors of $\widehat{Q}$ such that $\langle \widehat{\cc}, \widehat{\dd} \rangle_{\widehat{Q}}=0$. For any pair of representations $(\widehat{X}, \widehat{Y}) \in \rep(\widehat{Q}, \widehat{\cc})\times \rep(\widehat{Q}, \widehat{\dd})$, consider the $K$-linear map
\begin{align*}
d^{\widehat{X}}_{\widehat{Y}}\colon \bigoplus_{i \in \widehat{Q}_0}\Hom_K(\widehat{X}(i),\widehat{Y}(i)) & \to \bigoplus_{b \in \widehat{Q}_1} \Hom_K(\widehat{X}(tb), \widehat{Y}(hb)) \\
(\varphi(i))_{i \in \widehat{Q}_0} &\mapsto  (\varphi(hb)\widehat{X}(b)-\widehat{Y}(b)\varphi(tb))_{b \in \widehat{Q}_1},
\end{align*}
which can be viewed as a square matrix since $\langle \widehat{\cc}, \widehat{\dd} \rangle_{\widehat{Q}}=0$. 

Given a representation $\widehat{V} \in \rep(\widehat{Q}, \widehat{\cc})$, the regular function
\begin{equation}
c^{\widehat{V}} \colon \rep(\widehat{Q}, \widehat{\dd}) \to K, \qquad c^{\widehat{V}}(\widehat{W}) = \det(d^{\widehat{V}}_{\widehat{W}}),
\end{equation}
turns out to be a semi-invariant on $\rep(\widehat{Q}, \widehat{\dd})$ of weight $\langle \widehat{\cc}, - \rangle_{\widehat{Q}}$. It is also called a \key{Schofield determinantal semi-invariant}. 

Let $\dd \in \ZZ^{Q_0}_{\geq 0}$ and $\widehat{\cc} \in \ZZ^{\widehat{Q}_0}_{\geq 0}$ be dimension vectors of $Q$ and $\widehat{Q}$, respectively, such that $\langle \widehat{\cc}, \widehat{\dd} \rangle_{\widehat{Q}}=0$. For a representation $\widehat{V} \in \rep(\widehat{Q}, \widehat{c})$, define
$$
\res_{A,\dd}(c^{\widehat{V}})=\restr{\left( c^{\widehat{V}} \circ \tau_{\dd}\right)}{\rep(A,\dd)} \qquad \text{and} \qquad
\res_{C}(c^{\widehat{V}})=\restr{\left( c^{\widehat{V}} \circ \tau_{\dd}\right)}{C}
$$
where $C$ is any subvariety of $\rep(A,\dd)$. 
Now, we are ready to state the following extension of the First Fundamental Theorem for semi-invariants of quivers \cite[Theorem 1]{DW1} or \cite[Theorem 2.3]{SVB}.

\begin{prop}\label{FFT-thm} Let $\theta \in \ZZ^{Q_0}$ be an integral weight and $\dd_1, \ldots, \dd_n \in \ZZ^{Q_0}_{\geq 0}$ dimension vectors with $\theta(\dd_1)=\ldots=\theta(\dd_n)=0$. Assume that $\SI(A,\dd_i)_{\theta} \neq 0$ for all $1 \leq i \leq n$.

\begin{enumerate}[(a)]
\item There exists a dimension vector $\widehat{\cc} \in \ZZ^{\widehat{Q}_0}_{\geq 0}$ such that $\langle \widehat{\cc}, \widehat{\dd}_i \rangle_{\widehat{Q}}=0$, $\forall 1 \leq i \leq n$, and
$$
\mathsf{span}_K\{ \res_{A,\dd_i}(c^{\widehat{V}}) \mid \widehat{V} \in \rep(\widehat{Q}, \widehat{\cc}) \}=\SI(A,\dd_i)_{\theta}, \forall 1 \leq i \leq n.
$$

\item Let $C_i \subseteq \rep(A,\dd_i)$ be irreducible $\GL(\dd_i)$-invariant subvarieties ($1 \leq i \leq n$) such that $\SI(C_i)_{\theta}\neq 0$, $\forall 1 \leq i \leq n$. Let $\widehat{\cc}$ be a dimension vector as in (a). Then there exists a nonempty open subset $\mathcal U \subseteq \rep(\widehat{Q}, \widehat{\cc})$ such that
$$
\mathsf{span}_K\{ \res_{C_i}(c^{\widehat{V}}) \mid \widehat{V} \in \mathcal U\}=\SI(C_i)_{\theta}, \forall 1 \leq i \leq n,
$$
where all the semi-invariants in the spanning sets above are nonzero.
\end{enumerate}
\end{prop}

\begin{remark} 
We point out that one can always construct spanning sets of Schofield determinantal semi-invariants for any single given space $\SI(A,\dd)_{\theta}$ (or $\SI(C)_{\theta}$) by working entirely within the category of representations of $(Q,\R)$. However, if the algebra $A$ is not acyclic, it is not clear how to come up with the analogue of the dimension vector $\widehat{\cc}$ and locus $\mathcal U$, unless one uses the Schofield double quiver $\widehat{Q}$.
\end{remark}

To prove Proposition \ref{FFT-thm}, we need the result of Domokos below. To state it, we recall the following iterative way of building the double quiver $\widehat{Q}$ of $Q$. For an arbitrary vertex $v \in Q_0$, define the quiver $Q^{v}$ by:
\begin{itemize}
\item $Q^v_0:=(Q_0\setminus \{v\}) \cup \{v_0, v_1\}$;
\item $Q^v_1:= \{a \in Q_1 \mid a \text{~is not incident to~} v\}  \cup \{\widehat{a} \colon v_0 \to ha \mid a \in Q_1 \text{~with~} ta=v \} \cup \{\widehat{a} \colon ta \to v_1 \mid a \in Q_1 \text{~with~} ha=v\} \cup \{c_v\colon v_0 \to v_1\}$.
\end{itemize}
If $a \in Q_1$, we denote the corresponding arrow in $Q_1^v$, which is either $a$ itself or $\widehat{a}$, by $a^v$.
If $\dd \in \ZZ_{\geq 0}^{Q_0}$ is a dimension vector, define $\dd^v \in \ZZ^{Q^v_0}_{\geq 0}$ by $\dd^v(x)=\dd(x)$ for $x \in Q_0$, with $x \neq v$, and $\dd^v(v_0)=\dd^v(v_1)=\dd(v)$. We have the closed embedding $\tau^v_{\dd}\colon \rep(Q,\dd) \to \rep(Q^v,\dd^v)$ defined by $\tau^v_{\dd}(X)(a^v)=X(a)$ for all $a \in Q_1$ and $\tau^v_{\dd}(X)(c_v)=\Id_{X(v)}$.
\smallskip

If $\theta \in \ZZ^{Q_0}$ and $n \in \ZZ$, define $\theta^{v,n} \in \ZZ^{Q^{v}_0}$ by $\theta^{v,n}(i)=\theta(i)$ for all $v \neq i \in Q_0$, and $\theta^{v,n}(v_0)=\theta(v)+n$, and $\theta^{v,n}(v_1)=-n$.

\begin{prop} (compare to \cite[Proposition 2.1]{Domo2}) \label{Domokos-prop} Let $\dd$ be a dimension vector of $Q$, let $\theta \in \ZZ^{Q_0}$ an integral weight, and $v \in Q_0$. Fix an arbitrary semi-invariant $f\in \SI(Q,\dd)_{\theta}$. Then there is an integer $N^{v}_{f}>0$ such that for any integer $n \geq N^v_f$, there exists a semi-invariant $f^v_n \in \SI(Q^v,\dd^v)_{\theta^{v,n}}$ with $f=f^v_n \circ \tau^v_{\dd}.$
\end{prop}

\begin{proof} For each arrow $a \in Q_1$ and pair of indexes $(i,j) \in \{1, \ldots, \dd(ha)\}\times \{1, \ldots, \dd(ta)\}$, denote the corresponding coordinate function in $K[\rep(Q,\dd)]$ by $T^a_{i,j}$. Equip $K[\rep(Q,\dd)]$ with the grading defined by $\deg (T^a_{i,j})=1$ if $ta=v$ and $0$ if $ta \neq v$. It is clear that the action of $\GL(\dd)$ on $K[\rep(Q,\dd)]$ preserves this grading. Consequently, we can write
$$
f=f_1+\ldots+f_l,
$$
where each $f_i$ is a homogeneous semi-invariant of weight $\theta$. Let us denote the degree of $f_i$ by $d_i$,  $\forall 1 \leq i \leq l$. It is proved in \cite[Proposition 2.1]{Domo2} that, for each $1 \leq i \leq l$, there exists a semi-invariant $f^v_{d_i} \in \SI(Q^v,\dd^v)_{\theta^{v,d_i}}$ such that $f_i=f^v_{d_i} \circ \tau^v_{\dd}$.

Now, for each positive integer $m$, we have the semi-invariant $\det^m_{c_v}\colon\rep(Q^v,\dd^v) \to K$ defined by sending $X \in \rep(Q^v, \dd^v)$ to $\det(X(c_v))^m$. The weight of this semi-invariant is $m$ at $v_0$, $-m$ at $v_1$, and zero at all other vertices.

Finally, setting $N^v_f:=d_1+\ldots +d_l$, we get that for every $n \geq N^v_f$:
\begin{itemize}
\item $f^v_n:=\sum_{i=1}^l f^v_{d_i}\cdot \det_{c_v}^{n-d_i} \in \SI(Q^v,\dd^v)_{\theta^{v,n}}$; \\

\item $f^v_n \circ \tau^v_{\dd}=f$.
\end{itemize}
This finishes the proof.
\end{proof}

\begin{proof}[Proof of Proposition \ref{FFT-thm}] $(a)$ Since $A$ is finite-dimensional, we know that any weight space of semi-invariants for $A$ is finite-dimensional. For each $i  \in \{1, \ldots, n\}$, choose a $K$-basis $F^i_1, \ldots, F^i_{m_i}$ for $\SI(A,\dd_i)_{\theta}$. Furthermore, since ${\rm char} K = 0$, taking invariants preserves surjectivity of $K$-algebra homomorphisms, so we know that there exist semi-invariants $f^i_1, \ldots, f^i_{m_i}$ in $\SI(Q,\dd_i)_{\theta}$ such that $F^i_l=\restr{f^i_l} {\rep(A,\dd_i)}$ for all $1 \leq l \leq m_i$ and $1 \leq i \leq n$.

After successively applying Proposition \ref{Domokos-prop} to all $f^i_l$, and the vertices of $Q$, one vertex at a time, we get a weight $\widetilde{\theta} \in \ZZ^{\widehat{Q}_0}_{\geq 0}$ and semi-invariants $\hat{f}^i_l \in \SI(\widehat{Q}, \widehat{\dd_i})_{\widetilde{\theta}}$ such that $f^i_l=\hat{f}^i_l \circ \tau_{\dd_i}$ for all $l$ and all $i$.
We know from \cite{IOTW} (see also \cite[Theorem 2.7]{CC8}) that there exist unique dimension vectors $\widehat{\cc}, \widehat{\ff} \in \ZZ^{\widehat{Q}_0}_{\geq 0}$ such that $\supp(\widehat{\cc}) \cap \supp(\widehat{\ff})=\emptyset$ and $\widetilde{\theta}=\langle \widehat{\cc}-\ddim P_{\widehat{\ff}}, -\rangle_{\widehat{Q}}$. Since $\SI(\widehat{Q}, \widehat{\dd_i})_{\widetilde{\theta}} \neq 0$ for all $1 \leq i \leq n$, we also have that:
\begin{itemize}
\item $\supp(\widehat{\dd_i}) \cap \supp(\widehat{\ff})=\emptyset, \forall 1 \leq i \leq n$;
\item $\SI(\widehat{Q}, \widehat{\dd_i})_{\langle \widehat{\cc}, - \rangle_{\widehat{Q}}} \neq 0, \forall 1 \leq i \leq n$.
\end{itemize} 
(Recall that $P_{\widehat{f}}:=\bigoplus_{x \in \widehat{Q}_0}P_x^{\widehat{f}(x)}$, where $P_x$ is the projective indecomposable representation of $\widehat{Q}$ at vertex $x$.)

It is now easy to check that $\restr{\widetilde{\theta}}{\supp(\widehat{\dd_i})}=\restr{\langle \widehat{\cc}, - \rangle_{\widehat{Q}}}{\supp(\widehat{\dd_i})}$ and hence $$\SI(\widehat{Q}, \widehat{\dd_i})_{\widetilde{\theta}}=\SI(\widehat{Q}, \widehat{\dd_i})_{\langle \widehat{\cc}, - \rangle_{\widehat{Q}}}, \forall 1 \leq i \leq n.$$
From the First Fundamental Theorem for semi-invariants for acyclic quivers (see \cite{DW1} or \cite{SVB}), we know that each $\hat{f}^i_l$ is a linear combination of semi-invariants of the form $c^{\widehat{V}}$ with $\widehat{V} \in \rep(\widehat{Q}, \widehat{\cc})$. Hence, we get that
$$
\mathsf{span}_K\{ \res_{A,\dd_i}(c^{\widehat{V}}) \mid \widehat{V} \in \rep(\widehat{Q}, \widehat{\cc}) \}=\SI(A,\dd_i)_{\theta}, \forall 1 \leq i \leq n.
$$

\bigskip
\noindent
$(b)$  For each $1 \leq i \leq n$, we know from $(a)$ that there exist $\widehat{V} \in \rep(\widehat{Q}, \widehat{\cc})$ and $M_i \in \rep(A, \dd_i)$ such that $c^{\widehat{V}}(\widehat{M_i}) \neq 0$.
Since we assume $\SI(C_i)_{\theta}\neq 0$, and the restriction map $\SI(A, \dd_i)_{\theta} \to \SI(C_i)_{\theta}$ is surjective, we may take $M_i \in C_i$ even.
Consequently,
$$
\mathcal U_i:=\{\widehat{V} \in \rep(\widehat{Q}, \widehat{\cc}) \mid c^{\widehat{V}}(\widehat{M_i})\neq 0 \text{~for some~} M_i \in C_i \}
$$ 
is a nonempty open subset of $\rep(\widehat{Q}, \widehat{\cc})$. Then $\mathcal U:=\bigcap_{i=1}^n \mathcal U_i$
is clearly a nonempty open subset of $\rep(\widehat{Q}, \widehat{\cc})$, and $\res_{C_i}(c^{\widehat{V}}) \neq 0$ for all $\widehat{V} \in \mathcal U$ and all $i$.

Now, for each $i \in \{1, \ldots, n\}$, let us choose $\widehat{V}^i_1, \ldots, \widehat{V}^i_{m_i} \in \rep(\widehat{Q}, \widehat{\cc})$ such that $\res_{C_i}(c^{\widehat{V}^i_j})$, $1 \leq j \leq m_i$, form a $K$-basis for $\SI(C_i)_{\theta}$. Then we can choose representations $M^i_j \in C_i$, $j \in \{1, \ldots, m_i\}$, such that the matrix $\left(c^{\widehat{V}^i_k}(\widehat{M^i_l})\right)$ is nonsingular. Next, consider the regular function
\begin{align*}
\varphi_i\colon&\mathcal{U}_i^{m_i} \to K \\
(\widehat{X}^i_1, \ldots, \widehat{X}^i_{m_i})&\to \det (\left(c^{\widehat{X}^i_k}(\widehat{M^i_l})\right)),
\end{align*}
and note that $\varphi_i^{-1}(K \setminus \{0\})$ is not empty.

It is now easy to see that for any chosen tuple $(\widehat{X}^i_1, \ldots, \widehat{X}^i_{m_i}) \in \varphi_i^{-1}(K \setminus \{0\}) \cap \mathcal U_i^{m_i}$, the semi-invariants $\res_{C_i}(c^{\widehat{X}^i_1})$, $\ldots$, $\res_{C_i}(c^{\widehat{X}^i_{m_i}})$ form a $K$-basis for $\SI(C_i)_{\theta}$; in particular, this completes the proof by showing that
$$
\mathsf{span}_K \{\res_{C_i}(c^{\widehat{V}}) \mid \widehat{V} \in \mathcal U \}=\SI(C_i)_{\theta}, \forall 1 \leq i \leq n. \qedhere
$$ 
%This finishes the proof.
\end{proof}

\subsection{Removing orbit closure summands}
The following reduction result is an adaptation of \cite[Lemma 5.1]{Bob5}  to the general case where $Q$ may have oriented cycles.  With Proposition \ref{FFT-thm} at our disposal, the arguments in \textit{ibid.} carry over. Nonetheless, we include the proof below for completeness and for the convenience of the reader.

\begin{lemma}(see also \cite[Lemma 5.1]{Bob5}) \label{lemma-reduction-Bob} Let $C$ be an irreducible $\GL(\dd)$-invariant closed subvariety of $\rep(A,\dd)$ with $C^{ss}_{\theta} \neq \emptyset$. If $C=\overline{C_1 \oplus C_2}$ for irreducible $\GL(\dd_i)$-invariant closed subvarieties $C_i \subseteq \rep(A,\dd_i)$, $i \in \{1,2\}$, with $C_2$ the orbit closure of a representation $M_2$, then
$$
\M(C)^{ss}_{\theta} \simeq \M(C_1)^{ss}_{\theta}.
$$
\end{lemma}

\begin{proof} Replacing $\theta$ with a positive multiple, which does not change the moduli space, we can assume without loss of generality that $\SI(C)_{\theta} \neq 0$. This implies that $\SI(C_1)_{\theta} \neq 0$ and $\SI(C_2)_{\theta} \neq 0$.
Now, let us consider the morphism
\begin{align*}
\varphi\colon&C_1 \longrightarrow C \\
&X \mapsto X\oplus M_2 .
\end{align*}
Note that for any weight $\sigma \in \ZZ^{Q_0}$, we know $\varphi^{*}(\SI(C)_{\sigma}) \subseteq \SI(C_1)_{\sigma}$ and that the restriction $\varphi^{*}_{\sigma}:=\restr{\varphi^*}{\SI(C)_{\sigma}}\colon\SI(C)_{\sigma} \to \SI(C_1)_{\sigma}$ is injective. The injectivity follows immediately from the fact that the $\GL(\dd)$-orbit of the image of $\varphi$ is dense in $C$.

Now, let $m \geq 1$ be an integer and set $\sigma:=m\theta$. We claim that $\varphi^*_{\sigma}$ is surjective as well. To prove this claim, choose an open subset $\emptyset \neq \mathcal U \subseteq \rep(\widehat{Q}, \widehat{\ee})$ as in Proposition \ref{FFT-thm}(b). Then, for any $\widehat{V} \in \mathcal U$, we have that $c^{\widehat{V}}(\widehat{M_2}) \neq 0$ and
$$
\varphi^*_{\sigma}(\res_{C}(c^{\widehat{V}}))=c^{\widehat{V}}(\widehat{M_2})\cdot \res_{C_1}(c^{\widehat{V}}).
$$
So, we get that
$$
\varphi^*_{\sigma}\left({1 \over c^{\widehat{V}}(\widehat{M_2})}\res_{C}(c^{\widehat{V}})\right)=\res_{C_1}(c^{\widehat{V}}).
$$
Using Proposition \ref{FFT-thm}{(b)}, we conclude that $\varphi^*_{\sigma}$ is surjective and hence an isomorphism. Taking the sum over all $m$, we get an isomorphism of homogeneous coordinate rings of $\M(C)^{ss}_{\theta}$ and $\M(C_1)^{ss}_{\theta}$, completing the proof.
\end{proof}

With this, we can continue the proof of our main theorem.

\begin{proof}[Proof of Theorem \ref{main-thm}(b)]
This follows from applying Lemma \ref{lemma-reduction-Bob} repeatedly.
\end{proof}

\section{The product decomposition} 
\subsection{Outline of proof completion} We simplify the notation by assuming that $$C=\widetilde{C} = \overline{C_1^{\oplus m_1} \oplus \cdots \oplus C_r^{\oplus m_r}}.$$ Now we can construct the morphism $\Psi$ in the statement of Theorem \ref{main-thm}(c). 
We have an equality $(C')^{ss}_{\theta}=(C_{1,\theta}^{ss})^{m_1}\times \cdots \times (C_{r,\theta}^{ss})^{m_r}$ 
since direct summands of a $\theta$-semistable representation are $\theta$-semistable.  The group $\prod_{i=1}^r \left( S_{m_i} \ltimes \PGL(\dd_i)^{m_i} \right)$ naturally acts on the right hand side, and the GIT-quotient by this action is precisely $S^{m_1}(\M(C_1)^{ss}_{\theta}) \times \ldots \times S^{m_r}(\M(C_r)^{ss}_{\theta})$. Furthermore, using the universal property of this quotient, we get the commutative diagram
\begin{equation}\label{eq:mainiso}
\vcenter{\hbox{\begin{tikzpicture}
\node (1) at (0,2) {$(C')^{ss}_\theta$};
\node (2) at (4,2) {$C^{ss}_\theta$};
\node (3) at (0,0) {$\prod_{i=1}^r S^{m_i}(\M(C_i)^{ss}_{\theta})$};
\node (4) at (4,0) {$\M(C)^{ss}_{\theta}$};
\path[->>]
(1) edge node[left] {$\pi'$} (3)
(2) edge node[right] {$\pi$} (4);
\path[right hook-latex]
(1) edge node[above] {$i$} (2);
\path[->]
(3) edge node[above] {$\Psi$} (4);
\end{tikzpicture}}}
\end{equation}
where the vertical maps are the quotient morphisms. The following proposition, which will be proved in the next subsection, gives us the essential properties of $\Psi$.

\begin{prop}\label{psi-properties} The morphism $\Psi$ is finite and birational.
\end{prop}

Assuming Proposition \ref{psi-properties}, we can finish proving our main theorem.
 
\begin{proof}[Proof of Theorem \ref{main-thm}(c)] Proposition \ref{psi-properties} shows that $\Psi$ is finite and birational. This, combined with the assumption that $\M(C)^{ss}_{\theta}$ is normal, implies that $\Psi$ is in fact an isomorphism of varieties. (This is a standard fact from algebraic geometry: the isomorphism property can be checked locally on the target space, then use that by definition a normal domain admits no nontrivial finite extensions within its field of fractions.)
\end{proof}

\begin{remark} In practice, one way to check that $\M(C)^{ss}_{\theta}$ is normal is to show that the semi-stable locus $(\overline{\bigoplus_{i=1}^r C_i^{m_i}})^{ss}_{\theta} $ is normal, after throwing away the $C_i$ which are orbit closures.
\end{remark}

\subsection{Proof of technical ingredients} For the remainder of the section, set
\begin{itemize}
\item $C':=\prod_{i=1}^r C_i^{m_i} \subseteq C$;

\item $G_{\theta}:=\ker(\chi_{\theta}) \leq \GL(\dd)$, with $\chi_{\theta}$ as in \eqref{eq:chitheta};

\item $G':=\prod_{i=1}^r \left( S_{m_i} \ltimes \GL(\dd_i)^{m_i} \right)\leq \GL(\dd)$, acting naturally on $C'$;

\item $G'_\theta := G' \cap G_{\theta}$, i.e. $G'_{\theta}$ is the kernel of the restriction of $\chi_{\theta}$ to $G'$.
\end{itemize}

It now remains to prove Proposition \ref{psi-properties} along with Proposition \ref{KXG-prop} below. It is easier to work with the affine quotients $C//G_\theta = \Spec(K[C]^{G_\theta})$ and $C'//G'_\theta = \Spec(K[C']^{G'_\theta})$. Restriction of invariant functions $\psi^*\colon K[C]^{G_\theta} \to K[C']^{G'_\theta}$ induces the morphism of affine varieties
$$
\psi\colon C'//G'_\theta \to C//G_{\theta}, \qquad \psi(\pi_C'(x))=\pi_C(x), \forall x \in C'
$$
where $\pi_{C'}:C' \to C'//{G'_{\theta}}$ and $\pi_C:C \to C//G_{\theta}$ are the quotient morphisms induced by the inclusions of the invariant rings.

We need the following assumptions, which result in no loss of generality since our moduli spaces are unchanged when replacing $\theta$ by any of its positive multiples.
\begin{itemize}
\item $\sum_{x \in Q_0} \theta(x)$ is an even number (this is essential for Proposition \ref{KXG-prop});
\item $\bigoplus_{m \geq 0} \SI(C)_{m\theta}$ is generated by semi-invariants of weight $\theta$ (this is very useful for Proposition \ref{psi-properties}).
\item no $C_i$ is an orbit closure (this convenient for both, and no loss of generality by (b) of Theorem \ref{main-thm}).
\end{itemize}

\begin{prop}\label{KXG-prop}
With $\theta$ as above, we have
\begin{equation}\label{eq:KXGtheta}
K[C']^{G'_\theta}=\bigoplus_{m \geq 0}\bigotimes_{i=1}^r S^{m_i}\left( \SI(C_i)_{m \theta} \right), 
\end{equation} 
i.e. the affine quotient variety $C'//G'_{\theta}$ is the affine cone over $\prod_{i=1}^r S^{m_i}(\M(C_i)^{ss}_{\theta})$.
\end{prop}

\begin{proof}[Proof of Proposition \ref{KXG-prop}]
Now, let $G'$ be as above and denote by $\chi$ the restriction of $\chi_{\theta}$ to $G'$, so that $G'_\theta=\ker(\chi)$. Then we have the weight space decomposition
\begin{equation}
K[C']^{G'_\theta}=\bigoplus_{m \in \ZZ} \SI(C',G')_{\chi^m},
\end{equation}
where $\SI(C',G')_{\chi^m}:=\{f \in K[C'] \mid g\cdot f=\chi^m(g)f, \forall g \in G' \}$ is the space of $G'$-semi-invariants on $C'$ of weight $\chi^m$, $m \in \ZZ$. 

To show the containment $\supseteq$ of \eqref{eq:KXGtheta}, we denote by $\chi_j$, $1 \leq j \leq r$, the rational character of $\GL(\dd_j)$ induced by $\theta$. Then, for an arbitrary element $n=\sigma \cdot g \in G'$, where $\sigma=\sigma_1 \times \ldots \times \sigma_r \in S_{m_1}\times \ldots \times S_{m_r}$ and $g=(g^j_i)_{1 \leq i \leq m_j, 1 \leq j \leq r}\in \prod_{j=1}^r \GL(\dd_j)^{m_j} \subseteq \GL(\dd)$, we calculate the character value
\begin{equation} \label{char-value-eq}
\chi(n)=\prod_{j=1}^r sgn(\sigma_j)^{ \sum_{x \in Q_0} \theta(x)} \cdot \prod_{j=1}^r \prod_{i=1}^{m_j} \chi_j(g^j_i)=\prod_{j=1}^r \prod_{i=1}^{m_j} \chi_j(g^j_i)
\end{equation}
(the second equality is using that $\sum_x \theta(x)$ is even). Next note that an arbitrary element of the right hand side of \eqref{eq:KXGtheta} is a $K$-linear combination of elements of the form $h_m^1 \otimes \ldots \otimes h_m^r$ with $h_m^j \in S^{m_j}(\SI(C_j)_{m \theta})$, $1 \leq j \leq r$, $m \geq 0$. But such elements, viewed as a regular functions on $C'$, are easily seen to be $G'_{\theta}$-invariant. Indeed, let us fix $m$ and $j$, and a $K$-basis $f^1, \ldots, f^N$ of $\SI(C_j)_{m\theta}$. Then we can write $$h^j_m=\sum_{l_1, \ldots, l_{m_j}=1}^N T_{l_1, \ldots, l_{m_j}} f^{l_1} \otimes \ldots \otimes f^{l_{m_j}}$$ for unique $T_{l_1, \ldots, l_{m_j}} \in K$ such that $T_{l_1, \ldots, l_{m_j}}=T_{l_{\nu(1)}, \ldots, l_{\nu(m_j)}}$ for any $\nu \in S_{m_j}$. We can now see that $n \cdot h^j_m=\prod_{i=1}^{m_j} \chi^m_j(g^j_i)h^j_m$, and consequently
$$
n \cdot (h_m^1 \otimes \ldots \otimes h_m^r)= \prod_{j=1}^r \prod_{i=1}^{m_j} \chi^m_j(g^j_i)   h_m^1 \otimes \ldots \otimes h_m^r=\chi^m(n)h_m^1 \otimes \ldots \otimes h_m^r.
$$
So the containment $\supseteq$ holds in \eqref{eq:KXGtheta} \footnote{We point out that the inclusion $\supseteq$ in \eqref{eq:KXGtheta} does not hold if $\sum_{x \in Q_0}\theta(x)$ is odd. Indeed, if that is the case, then one can easily find elements $n \in G'_\theta$ such that $n\cdot f=-f$ for any $f \in \bigotimes_{i=1}^r S^{m_i}\left( \SI(C_i)_{\theta} \right)$, viewed as a regular function on $C'$.}. For the other containment $\subseteq$, consider $L:=\prod_{j=1}^r S_{m_j} \ltimes \Ker(\chi_j)^{m_j}$. Since $L \leq G'_\theta$ by \eqref{char-value-eq}, we have that
\begin{equation}\label{eq:KXGtheta2}
K[C']^{G'_\theta} \subseteq K[C']^L=\bigotimes_{j=1}^r \left(\underbrace{K[C_j]^{\Ker(\chi_j)}\otimes \ldots \otimes K[C_j]^{\Ker(\chi_j)}}_{m_j\text{~times}} \right)^{S_{m_j}}.
\end{equation}
Next, note that each $\chi_j$ is not the trivial character of $\GL(\dd_j)$ since $C_j$ contains a $\theta$-stable representation and is not an orbit closure. 
Therefore, we have that $K[C_j]^{\Ker(\chi_j)}=\bigoplus_{m \geq 0} \SI(C_j)_{m\theta}$.
Applying this to each term on the right hand side of \eqref{eq:KXGtheta2}, we can write
\begin{equation} \label{wt-space-dec-eqn}
K[C']^L=\bigotimes_{j=1}^r \left(\underbrace{\left(\bigoplus_{m \geq 0} \SI(C_j)_{m\theta}\right)\otimes \ldots \otimes \left(\bigoplus_{m \geq 0} \SI(C_j)_{m\theta}\right)}_{m_j\text{~times}} \right)^{S_{m_j}}.
\end{equation}
Now, a simple check shows that for each $m \in \ZZ$, the subspace of the right hand side of \eqref{wt-space-dec-eqn}  consisting of the $G'$-semi-invariants on $C'$ of weight $\chi^m$ is precisely
\begin{equation} \label{wt-space-eqn}
\bigotimes_{j=1}^r \left(\underbrace{\SI(C_j)_{m\theta}\otimes \ldots \otimes \SI(C_j)_{m\theta}}_{m_j\text{~times}} \right)^{S_{m_j}}.
\end{equation}
Finally, combining \eqref{eq:KCGtheta}, \eqref{eq:KXGtheta2}, \eqref{wt-space-dec-eqn}, and \eqref{wt-space-eqn}, we obtain

$$
\SI(C',G')_{\chi^m}=
\begin{cases}
\bigotimes_{j=1}^r S^{m_j}(\SI(C_j)_{m\theta}) & \text{~if~} m\geq 0,\\
0 & \text{~otherwise.}
\end{cases}
$$
This finishes the proof of our first technical proposition. 
\end{proof}

To prove Proposition \ref{psi-properties}, we need the following result in invariant theory which may be known to experts, but for which we are unaware of a suitable reference.  It seems to be a relative version of \cite[Lemma 2.4.5]{Harm-book-02}.

\begin{lemma} \label{finite-map-lemma} Let $H'$ and $H$ be a linearly reductive groups with $H' \leq H$ and $V$ a finite-dimensional rational $H$-module. Let $X$ be an affine $H$-subvariety of $V$ and $X'$ an affine $H'$-subvariety of $X$ such that $0 \in X'$, and denote by $\pi_{X'}\colon X' \to X'//H'$ and $\pi_X\colon X \to X//H$ the quotient morphisms. Denote the image of $0 \in V$ through the two morphisms by the same symbol $\mathbf{0}$.

Let $\psi: X'//H' \to X//H$ be the morphism of varieties induced by the restriction homomorphism $\psi^{*}:K[X]^H \to K[X']^{H'}$.
Suppose furthermore that $\psi$ is $K^*$-equivariant for some torus $K^*$ acting on these spaces  which fixes $\mathbf{0}$, and that the induced grading on coordinate rings is supported in nonnegative degrees, with the maximal ideals of functions vanishing at $\mathbf{0}$, say $\m \subset K[X]^H$ and $\m' \subset K[X']^{H'}$, being contained in the positive degree parts of these rings.

Under these assumptions, if $\psi^{-1}(\mathbf{0})=\mathbf{0}$, then $\psi$ is a finite morphism.
\end{lemma} 

\begin{proof}
Since $H', H$ are linearly reductive, both $K[X]^H$ and $K[X']^{H'}$ are finitely generated $K$-algebras (see for example \cite[Cor. 2.2.11]{Harm-book-02}), thus Noetherian, so that $\m$ and $\m'$ are finitely generated ideals within these algebras.
We want to show that the morphism of varieties $\psi$ is finite, which by definition means that the extension $\psi^*(K[X]^H) \subseteq K[X']^{H'}$ is module finite.

The assumption that $\psi^{-1}(\mathbf{0}) = \{\mathbf{0}\}$ translates to the equality of vanishing sets
 \begin{equation}\label{eq:vanishingsets}
 V(\psi^*(\m))=\{\mathbf{0}\} = V(\m')
 \end{equation}
 in $X'//H'$.
Let $I=\psi^*(\m)K[X']^{H'}$ be the ideal generated by $\psi^*(\m)$ in $K[X']^{H'}$.  By Hilbert's Nullstellensatz and the fact that $\m'$ is maximal, \eqref{eq:vanishingsets} implies that $\sqrt{I} = \m'$.
Since $\m'$ is finitely generated, some power of $\m'$ is contained in $I$, say $(\m')^N \subseteq I$.

Write $\m = \langle f_1, \dotsc, f_r\rangle$ and $\m' = \langle h_1, \dotsc, h_s\rangle$ where each $f_i$ and $h_i$ is homogeneous of positive degree; we have that each $\psi^*(f_i)$ is homogeneous in $K[X']$ as well since $\psi$  is $K^*$-equivariant.  
Now we claim that the set $\mathcal{S} = \setst{h_1^{i_1} \cdots h_s^{i_s}}{\forall j:  0 \leq i_j < N}$ generates $K[X']^{H'}$ as a $\psi^*(K[X]^H)$-module. Indeed, since $K[X']^{H'} = K[h_1, \dotsc, h_s]$, it is enough to show that an arbitrary monomial $h_1^{i_1} \cdots h_s^{i_s}$ is in the $\psi^*(K[X]^H)$-span of $\mathcal{S}$. Suppose not, for contradiction, and take a minimal degree counterexample;   
without loss of generality assume that $i_1 \geq N$.   Since $(\m')^N \subseteq I$,  we can rewrite $h_1^{i_1} = \sum_j \alpha_j \psi^*(f_j)$ for some $\alpha_j \in K[X']^{H'}$, each of degree smaller than the degree of $h_1^{i_1}$ since each $\psi^*(f_j)$ is of positive degree.  By the minimality assumption, we have for each $j$ that the monomial $\alpha_j h_2^{i_2} \cdots h_s^{i_s}$ is in the $\psi^*(K[X]^H)$-span of $\mathcal{S}$.  So substitution shows that the original monomial was as well, a contradiction which completes the proof.
\end{proof}

\begin{proof}[Proof of Proposition \ref{psi-properties}] First we will show that $\Psi$ is birational by checking that $\Psi$ is dominant and injective on a dense subset. The fact that $\Psi$ is dominant follows immediately from the definition of $\theta$-stable decomposition. 

Now we show that $\Psi$ is injective on a dense subset of $(C')^{ss}_{\theta}//PG'$ where $PG':=\prod_{i=1}^r \left( S_{m_i} \ltimes \PGL(\dd_i)^{m_i} \right)$. For each $i$, let
\[
C_i^\circ = C_i \setminus \left( \bigcup_{\substack{\dd_{i'} = \dd_i\\ C_{i'} \nsupseteq C_i}} C_{i'} \right) =  \bigcap_{\substack{\dd_{i'} = \dd_i\\ C_{i'} \nsupseteq C_i}} \left(C_i \setminus C_{i'} \right).
\]
Note that since each $C_i$ is closed and irreducible, $C_i \cap C_{i'}$ has smaller dimension than $C_i$ whenever $C_{i'} \nsupseteq C_i$, so each such $C_i \setminus C_{i'}$ is open and dense in $C_i$.  Since $C_i^\circ$ is a finite intersection of such subsets, it is open and dense in $C_i$ as well.  Then $U = \prod_i ((C_i^\circ)^s_{\theta})^{m_i}$ is open and dense in $C'$, and $\pi'(U)$ is dense in $(C')^{ss}_{\theta}//PG'$.

Since the $C_i$ are assumed to be distinct, for each pair $i \neq i'$ we must have that either $C_{i'} \nsupseteq C_i$ or $C_i \nsupseteq C_{i'}$, so by construction, we have $C_i^\circ \cap C_{i'}^\circ = \emptyset$ whenever $\dd_i = \dd_{i'}$ but $i \neq i'$.  Restricting to stable representations now, we find that $\Hom_A(M, N) = 0 =\Hom_A(N, M)$ whenever $M \in (C_i^\circ)^s_{\theta}$ and $N \in (C_{i'}^\circ)^s_{\theta}$ with $i \neq i'$, since stable representations are simple objects in the (full) category of semistable representations of $A$.

Now let $M, N \in U$ be such that $\Psi(\pi'(M))=\Psi(\pi'(N))$, so by definition $\pi(M)=\pi(N)$ in $C^{ss}_{\theta}//\PGL(\dd)$. Since $M$ and $N$ are $\theta$-polystable, their $\PGL(\dd)$-orbits are closed in $C^{ss}_{\theta}$ and so $M$ and $N$ are in the same $\PGL(\dd)$-orbit, which is the same as saying that $gM = N$ for some $g \in \GL(\dd)$. In particular, they are isomorphic representations of $A$.  We will use this to show that $g \in G'$ which will imply that $\pi'(M) = \pi'(N)$ in $(C')^{ss}_{\theta}//PG'$ and complete our proof that $\psi$ is injective on $\pi'(U)$.

Since $M$ and $N$ are $\theta$-polystable, they are semi-simple objects in the category of $\theta$-semistable representations of $A$, which greatly restricts the possible isomorphisms between them.  Write
\begin{equation}
M = \bigoplus_{i=1}^r \bigoplus_{j=1}^{m_i} M^j_i \qquad \text{and} \qquad N = \bigoplus_{i=1}^r \bigoplus_{j=1}^{m_i} N^j_i
\end{equation}
where each $M^j_i, N^j_i \in (C_i^\circ)^s_{\theta}$.  We know that $\Hom_A(M^j_i, N^{j'}_{i'}) = 0$ if $i \neq i'$ by the observation two paragraphs above. So for any isomorphism $M\xrightarrow{\varphi} N$, there exist  permutations $\sigma_1, \dotsc, \sigma_r$ such that $\varphi$ is a direct sum of isomorphisms of the form $M^j_i \simeq N^{\sigma_i(j)}_i$.
These isomorphisms are realized by elements $g^j_i \in \GL(\dd_j)$ satisfying $g^j_i M^j_i  = N^{\sigma_i(j)}_i$ for all $i, j$.  Therefore, our $g \in \GL(\dd)$ above is of the form
\begin{equation}
g = (\sigma_j, g_1^j, \dotsc, g_{m_j}^j)_{j=1}^r \in G'
\end{equation}
so we see that $M$ and $N$ are indeed in the same $G'$, equivalently, $PG'$ orbit. We have just proved that $\Psi$ is injective on a dense subset which, combined with $\Psi$ being dominant, implies that $\Psi$ is birational.

To prove that $\Psi$ is finite, we first show that $\psi$ is finite.  We do this by applying Lemma \ref{finite-map-lemma} with $H'=G'_\theta$ and $H=G_\theta$ acting on $X'=C'$ and $X=C$ inside the rational $G_\theta$-module $\rep_Q(\dd)$. Note that $G'/G'_\theta$ and $\GL(\dd)/G_\theta$, which act on $C'//G'_{\theta}$ and $C//G_{\theta}$, respectively, can be identified with $K^*$, making $\psi$ a $K^*$-equivariant morphsim.
Therefore, to show that $\psi$ is finite, it is enough to check that $\psi^{-1}(\mathbf{0})=\{\mathbf{0}\}$. This is equivalent to checking that, for $M \in C'$ with $0 \notin \overline{G'_\theta \cdot M}$, we have that $0 \notin \overline{G_{\theta} \cdot M}$, or in other words that $M$ is $\theta$-semistable. For such an $M$, write 
$$M=\bigoplus_{i=1}^r M^1_i\oplus \ldots \oplus M^{m_i}_i,$$
where $M^j_i \in C_i$ for all $1 \leq i \leq r$ and $1 \leq j \leq m_i$. We claim that each $M^j_i$ is $\theta$-semistable. For a contradiction, let assume that at least one of the $M^j_i$'s, say $M^1_1$, is not $\theta$-semistable. Denote by $\chi_1$ the rational character of $\GL(\dd_1)$ induced by $\theta$.  Then there exists a $1$-parameter subgroup $\lambda' \colon K^* \to \Ker(\chi_1)$ such that 
$$\lim_{t \to 0} \lambda'(t)M^1_1=0_{\dd_1},$$ 
where $0_{\dd_1}$ is the zero element of $\rep(Q,\dd_1)$. Viewing $\Ker(\chi_1)$ as a subgroup of $G'_\theta$, we get a $1$-parameter subgroup $\lambda \colon K^* \to G'_\theta$ such that 
$$
\lim_{t \to 0} \lambda(t)M=0_{\dd_1} \oplus M_1^2 \oplus \ldots \oplus M_r^{m_r}.
$$
In particular, this shows that
\begin{equation}\label{eq-1-closure}
0_{\dd_1} \oplus M_1^2 \oplus \ldots \oplus M_r^{m_r} \in \overline{G'_\theta \cdot M}.
\end{equation}
Now, since $A$ is finite-dimensional, we know that $0_{\dd_i} \in \overline{\GL(\dd_i) M_i^j}$ for all $1 \leq i \leq r$, $1 \leq j \leq m_i$. So, there are $1$-parameter subgroups $\lambda^j_i \colon K^* \to \GL(\dd_i)$ with $1 \leq j \leq m_i$, $1 \leq i \leq r$, and $(i,j)\neq (1,1)$, such that 
$$
\lim_{t \to 0} \lambda^j_i(t) M^j_i = 0_{\dd_i}. 
$$
Given these $1$-parameter subgroups, define a 1-parameter subgroup $\lambda^1_1$ of $\GL(\dd_1)$, given by the formula
\begin{equation}
\lambda^1_1(t)=\left( {\rm diag}\left(\prod \limits_{\substack{i,j \\ (i,j) \neq (1,1)}}\det(\lambda^j_i(t)(x))^{-1}, 1, \dotsc, 1\right)\right)_{x \in Q_0} .
\end{equation}
Then the 1-parameter subgroup $\mu$ of $\GL(\dd_1)^{m_1} \times \ldots \times \GL(\dd_r)^{m_r}$ given componentwise by $\mu(t) = (\lambda^j_i(t))_{i,j}$ satisfies $\chi_{\theta}(\mu(t))=1, \forall t \in K^*$.
So $\mu$ is in fact a $1$-parameter subgroup of $G'_\theta$. Moreover, we have that
$$
\lim_{t \to 0}\mu(t)\cdot (0_{\dd_1}\oplus M_1^2 \oplus \ldots \oplus M_r^{m_r})=0_{\dd_1}^{m_1}\oplus \ldots \oplus 0_{\dd_r}^{m_r},
$$
which shows that 
\begin{equation}\label{eq-2-closure}
0_{\dd_1}^{m_1}\oplus \ldots \oplus 0_{\dd_r}^{m_r} \in \overline{G'_\theta \cdot \left( 0_{\dd_1}\oplus M_1^2 \oplus \ldots \oplus M_r^{m_r}\right)}.
\end{equation}
From $(\ref{eq-1-closure})$ and $(\ref{eq-2-closure})$, we get that $0 \in \overline{G'_\theta \cdot M}$ (contradiction). Hence, each direct summand $M^j_i$ of $M$ is $\theta$-semistable, and so $M$ is semistable, i.e. $0 \notin \overline{G_{\theta} \cdot M}$. We have just proved that $\psi^{-1}(\mathbf{0})=\mathbf{0}$ which implies that $\psi$ is finite by Lemma \ref{finite-map-lemma}. 

Finally, let us quickly explain how the finiteness of $\psi$ implies that of $\Psi$. For this, we work with the following local description of $\Psi$. Denote by $R$ and $R'$ the homogeneous coordinate rings of the projcetive varieties $\M(C)^{ss}_{\theta} $ and $S^{m_1}(\mathcal{M}(C_1)^{ss}_{\theta}) \times \ldots \times S^{m_r}(\mathcal{M}(C_r)^{ss}_{\theta})$, respectively. Then, for any non-zero semi-invariant $f \in \SI(C)_{m\theta}$ with $m \geq 1$, the images through $\pi$ and $\pi'$ of the principal open subsets defined by $f$ are $\Spec((R[{1 \over f}])^{\GL(\dd)})$ and $\Spec((R'[{1 \over f'}])^{G'})$, respectively, where $f'$ is the restriction of $f$ to $C'$. Moreover, the preimage of $\Spec((R[{1 \over f}])^{\GL(\dd)})$ under $\Psi$ is precisely $\Spec((R'[{1 \over f'}])^{G'})$. It is immediate to see that $(R[{1 \over f}])^{\GL(\dd)}=R[{1 \over f}]_0$ and $(R'[{1 \over f'}])^{G'}=R'[{1 \over f'}]_0$. It now follows that the restriction of $\Psi$ to $\Spec((R'[{1 \over f'}])^{G'})$, taking values in $\Spec((R[{1 \over f}])^{\GL(\dd)})$, is a finite morphism of affine varieties since $\psi^*$ induces a finite homomorphism of rings $R[{1 \over f}]_0 \to R'[{1 \over f'}]_0$. So we see that $\Psi$ is a finite morphism, which completes the proof.
\end{proof}

\section{Example applications}
Our first example, in which the algebra $A$ is of wild representation type, demonstrates that $\Psi$ may fail to be an isomorphism and fail to be bijective even.  

\begin{example}\label{ex:notiso}
Let $Q$ be the following quiver with $Q_0=\{1, 2, 3\}$, four arrows from 1 to 2, and four arrows from 2 to 3.  By a standard abuse of notation, we refer to each collection of four arrows as $\{x_0, x_1, x_2, x_3\}$ as it simplifies the notation for the relations.
\[
Q:= \vcenter{\hbox{\begin{tikzpicture}[point/.style={shape=circle,fill=black,scale=.5pt,outer sep=3pt},>=latex]
   \node[point,label=above:{1}] (1) at (-4,0) {};
   \node[point,label=above:{2}] (2) at (0,0) {};
   \node[point,label=above:{3}] (3) at (4,0) {};
\path[->]
(1) edge node[midway,above] {$x_0, x_1, x_2, x_3$} (2)
(2) edge node[midway,above] {$x_0, x_1, x_2, x_3$} (3);
   \end{tikzpicture}}}
\]
For relations we take $\mathcal{R}=\{x_ix_j=x_jx_i\}_{i,j} \cup \{x_1x_2=0\}$.  (Note that the algebra $A=KQ/\langle R \rangle$ is wild since the subcategory of its representations which are zero at vertex 3 is equivalent to the category of representations of the 4-Kronecker quiver.)
Take dimension vector $\dd = (1,1,1)$ and denote a representation by
\begin{equation}\label{eq:repM}
M= \vcenter{\hbox{\begin{tikzpicture}[point/.style={shape=circle,fill=black,scale=.5pt,outer sep=3pt},>=latex]
   \node[point,label=above:{1}] (1) at (-4,0) {};
   \node[point,label=above:{2}] (2) at (0,0) {};
   \node[point,label=above:{3}] (3) at (4,0) {};
\path[->]
(1) edge node[midway,above] {$(a_0, a_1, a_2, a_3)$} (2)
(2) edge node[midway,above] {$(b_0, b_1, b_2, b_3)$} (3);
   \end{tikzpicture}}}
\end{equation}
where $a_i \in K$ is the entry of the $1\times 1$ matrix over arrow $1 \xrightarrow{x_i} 2$, and $b_i \in K$ is the entry of the $1\times 1$ matrix over arrow $2 \xrightarrow{x_i} 3$.
For weight $\theta = (2, -1, -1)$, the $\theta$-semistable representations are exactly those for which at least one $a_i$ and at least one $b_i$ is nonzero, 
and these are all in fact $\theta$-stable.
The relations $x_ix_j=x_jx_i$ then can be interpreted as saying that $(a_0, a_1, a_2, a_3), (b_0, b_1, b_2, b_3) \in K^4\setminus\{(0,0,0,0)\}$ are representatives of the same point in $\mathbb{P}^3$.

It is then straightforward to check that the map sending the isomorphism class of a representation $M$ as in \eqref{eq:repM} to the point $[a_0:a_1:a_2:a_3] \in \mathbb{P}^3$ gives an isomorphism
\begin{equation}\label{eq:m1iso}
\M(A, \dd)^{ss}_\theta = \M(A, \dd)^{s}_\theta \cong \setst{[a_0:a_1:a_2:a_3] \in \mathbb{P}^3}{a_1a_2 = 0}.
\end{equation}
Therefore, $\M(A, \dd)^{ss}_\theta$ is a two dimensional variety, with two irreducible components $Y_1 = \{a_1 = 0\}$ and $Y_2 = \{a_2 = 0\}$.  
% which intersect in the dimension 1 subvariety $Y_1 \cap Y_2 = \{a_1 = a_2 = 0\}$.  
Let $C_1, C_2 \subseteq \rep(A,\dd)$ be the corresponding irreducible components.

Now consider the closed subvariety $\M(\overline{C_1 \oplus C_2})^{ss}_{\theta} \subseteq \M(A, (2,2,2))^{ss}_{\theta}$ and the morphism
\[
\Psi\colon  \mathcal{M}(C_1)^{ss}_{\theta} \times \mathcal{M}(C_2)^{ss}_{\theta} \to \M(\overline{C_1 \oplus C_2})^{ss}_{\theta}
\]
of Theorem \ref{main-thm}.
Utilizing the isomorphism \eqref{eq:m1iso}, we claim that any two points in the domain of $\Psi$ of the form
\begin{equation}\label{eq:doublefiber}
p_1=[a_0:0:0:a_3] \times [a'_0:0:0:a'_3] \quad \text{and} \quad p_2=[a'_0:0:0:a'_3] \times [a_0:0:0:a_3]
\end{equation}
satisfy $\Psi(p_1)=\Psi(p_2)$. We can lift to representatives of $\Psi(p_1), \Psi(p_2)$, call them  $M_1, M_2 \in \rep(A,(2,2,2))$.  By construction, we have each $M_i \simeq M'_i \oplus M''_i$ where all $M'_i, M''_i \in C_1 \cap C_2$.  The form of the points $p_1, p_2$ shows that $M'_1 \simeq M''_2$ and $M'_2 \simeq M''_1$, so we have $M_1 \simeq M_2$, which implies $\Psi(p_1)=\Psi(p_2)$.
So we see that the points of $\M(\overline{C_1 \oplus C_2})^{ss}_{\theta}$ with disconnected preimage arise from $\theta$-stable points in the intersection $C_1 \cap C_2$ of the distinct components of the $\theta$-stable decomposition of $\overline{C_1 \oplus C_2}$.
\end{example}

In the following example, we apply the main result to illustrate how different isomorphism types of $\M(C)^{ss}_\theta$ may arise even when the $\M(C_i)^{ss}_\theta$ are all isomorphic to $\PP^1$.
%but the main theorem can be applied to see that for $\theta$-polystable representations the 
%in which we can see $\PP^1$ families of $\theta$-stable components combining to build $\theta$-semistable moduli spaces in different ways. \ryan{clarify this: specifically how combining  ``1-parameter families'' gives different geometry depending on the relation of these families.}

\begin{example}\label{ex:families}
Consider the special biserial algebra $A=KQ/\langle \mathcal{R} \rangle$ given by the following quiver with relations, and the weight $\theta = (2, -1, -1)$.
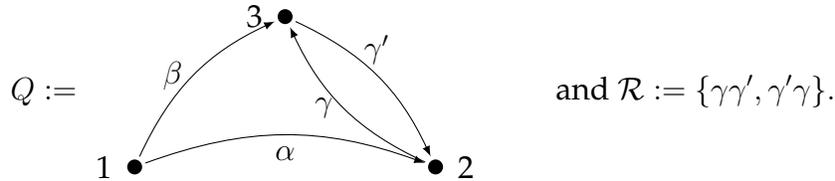
\[
Q:= \vcenter{\hbox{\begin{tikzpicture}[point/.style={shape=circle,fill=black,scale=.5pt,outer sep=3pt},>=latex]
   \node[point,label=left:{1}] (1) at (-2,0) {};
   \node[point,label=right:{2}] (2) at (2,0) {};
   \node[point,label=left:{3}] (3) at (0,2) {};
\path[->,bend left=20]
(1) edge node[midway, below] {$\alpha$} (2)
(2) edge node[midway, left] {$\gamma$} (3);
\path[->,bend left=20]
(1) edge node[midway, left] {$\beta$} (3)
(3) edge node[midway, above] {$\gamma'$} (2);
   \end{tikzpicture}}} \qquad \text{~and~} 
   \mathcal R:=\{\gamma \gamma', \gamma' \gamma\}.
 \]
For $\dd=(1,1,1)$, the representation variety $\rep(A,\dd) = \setst{(a, b, c, c') \in \mathbb{A}^4}{cc'= 0}$ has two irreducible components $C_1 = \{c = 0\}$ and $C_2 = \{c'=0\}$, both of which are $\theta$-stable.  It is straightforward to check that $\M(C_1)^{ss}_{\theta} \simeq \PP^1$ with homogeneous coordinates $(a:bc')$ and $\M(C_2)^{ss}_{\theta} \simeq \PP^1$ with homogeneous coordinates $(b:ac)$.  Their general elements are band modules, with points at 0 and $\infty$ being string modules.

For $\dd=(2,2,2)$, the representation variety
\[
\rep(A,\dd) = \setst{(A,B,C,C') \in (\Mat_{2 \times 2}(K))^4}{CC' = C'C=0}
\]
has three irreducible components, all of which are $\theta$-semi-stable; in fact, since $\ext^1_A(C_i, C_j) = 0$ for $i, j \in \{1, 2\}$ above, these components are just $\overline{C_1^{\oplus 2}}, \overline{C_1\oplus C_2}$, and $\overline{C_2^{\oplus 2}}$ which are  all normal. Since the $m$-th symmetric powers of $\PP^1$ is $\PP^m$, we get that
\[
\M(\overline{C_i^{\oplus 2}})^{ss}_{\theta} \simeq \PP^2, \ i=1,2,\quad \text{while} \quad \M(\overline{C_1\oplus C_2})^{ss}_\theta \simeq \PP^1 \times \PP^1.
\]
It is interesting to note that there is an elementary family of band modules
\[
\vcenter{\hbox{\begin{tikzpicture}[point/.style={shape=circle,fill=black,scale=.5pt,outer sep=3pt},>=latex]
   \node (1) at (-2,0) {$K^2$};
   \node (2) at (2,0) {$K^2$};
   \node (3) at (0,2) {$K^2$};
\path[->,bend left=20]
(1) edge node[midway, below] {$\left[\begin{smallmatrix} 1 & 0\\0 & 1\end{smallmatrix}\right]$} (2)
(2) edge node[midway, left] {$\left[\begin{smallmatrix} 0 & 1 \\0 & 0\end{smallmatrix}\right]$} (3);
\path[->,bend left=20]
(1) edge node[midway, left] {$\left[\begin{smallmatrix} 1 & 0\\0 & 1\end{smallmatrix}\right]$} (3)
(3) edge node[midway,right] {$\left[\begin{smallmatrix} 0 & \lambda \\0 & 0\end{smallmatrix}\right]$} (2);
   \end{tikzpicture}}} \qquad \lambda \in K.
\]
of this dimension vector which are semistable but not stable.  Every representation in this family has the same $\theta$-stable composition factors, namely the string modules $(0,1,0,1) \in C_1$ and $(1,0,1,0) \in C_2$ each with multiplicity 2, so this entire family is only represented by a single point in $\M(\overline{C_1\oplus C_2})^{ss}_\theta$.

This general pattern continues: given any irreducible component $\M(C)^{ss}_{\theta}$ of $\M(A,\dd)^{ss}_{\theta}$ for any $\dd$, by Theorem \ref{main-thm}(a) we may restrict our attention to the direct sum of the $\theta$-stable components of a $\theta$-stable decomposition of $C$. Then, applying Theorem \ref{main-thm}(b) to remove orbit closures from the $\theta$-stable decomposition, we may assume that $C=\overline{C_1^{\oplus m_1} \oplus C_2^{\oplus m_2}}$ for some $m_1, m_2 \in \ZZ_{\geq 0}$.  These are known to be normal, for example, by combining \cite[Theorem~11.3]{Lusztig} and  \cite[Theorem~8]{Faltingsloop}.  Therefore, Theorem \ref{main-thm}(c) gives that 
$$\M(\overline{C_1^{\oplus m_1} \oplus C_2^{\oplus m_2}})^{ss}_\theta \simeq \PP^{m_1} \times \PP^{m_2}.$$
This example is generalized to arbitrary special biserial algebras in \cite{CarChiKinWey2017}.
\end{example}

We end with a result on birational classification of moduli spaces for a certain class of algebras.  
These \emph{Schur-tame algebras} are, informally, generalizations of tame algebras which only require that families of nonisomorphic \emph{Schur} representations of the same dimension are at most one-dimensional.  
Moduli spaces of Schur-tame algebras were studied by the first author and A. Carroll in \cite{CC15}; they have previously been studied by L. Bodnarchuk and Y. Drozd in \cite{BodDro}.

\begin{corollary}\label{cor:tamerational}
Suppose that $A$ is a Schur-tame algebra (for example, a tame algebra).  Then the irreducible components of any moduli space $\M(A,\dd)^{ss}_{\theta}$ are rational varieties.
\end{corollary}
\begin{proof} Let $Y$ be an irreducible component and write $Y=\M(C)^{ss}_{\theta}$ for some irreducible component $C \subseteq \rep(A,\dd)$ with $C^{ss}_{\theta} \neq \emptyset$.
Retaining the notation of Theorem \ref{main-thm}, first consider each moduli space $\M(C_i)^{ss}_\theta$ of the components of the $\theta$-stable decomposition. It is shown for Schur-tame algebras in \cite[Proposition 12]{CC15} that such a moduli space $\M(C_i)^{ss}_{\theta}$ is always either a point or a rational projective curve. Since any symmetric power of a rational variety is rational and $\Psi$ is birational, $\M(C)^{ss}_\theta$ is rational as well.
\end{proof}

\bibliography{biblio-decomp}\label{biblio-sec}

\end{document}